\documentclass{article}

\usepackage{spconf}
\usepackage{amsbsy}
\usepackage{amsmath}
\usepackage{amsthm}
\usepackage{amssymb}
\usepackage{latexsym}
\usepackage{mathtools}
\usepackage{cite}
\usepackage{cite}
\usepackage[top=1in,bottom=.75in,right=0.75in,left=.75in]{geometry}

\newtheorem{theorem}{Theorem}

\newtheorem{assumption}{Assumption}

\newtheorem{lemma}{Lemma}
\newtheorem{corollary}{Corollary}

  \def\vw{{\bf w}} \def\vx{{\bf x}}
\def\vy{{\bf y}}

\def \bxi{\boldsymbol{\zeta}}

\def \bxi{\boldsymbol{\xi}}

\def \e{\varepsilon}
\def \R{\mathbb{R}}
\def \dx{\,dx}
\def \E{\mathbb{E}}
\def \P{\mathbb{P}}

\def \ones{\textbf{1}}
\def \lambdaMin{\lambda_{\text{min}}}
\def \P{\mathbb{P}}

\def \calF{\mathcal{F}}

\mathtoolsset{showonlyrefs}

\title{On Distributed Stochastic Gradient Algorithms\\ for Global Optimization}

\name{Brian Swenson, Anirudh Sridhar and H. Vincent Poor
                       \thanks{This work was partially supported by the Air Force Office of Scientific Research under MURI Grant FA9550-18-1-0502. }}
                 \address{Department of Electrical Engineering, Princeton University\\
                 Email: \{bswenson, anirudhs, poor\}@princeton.edu}

\begin{document}
\maketitle
\pagenumbering{gobble}

\renewcommand{\thefootnote}{\fnsymbol{footnote}}

\renewcommand{\thefootnote}{\arabic{footnote}}

\begin{abstract}
The paper considers the problem of network-based computation of global minima in smooth nonconvex optimization problems. It is known that distributed gradient-descent-type algorithms can achieve convergence to the set of global minima by adding slowly decaying Gaussian noise in order escape local minima. However, the technical assumptions under which convergence is known to occur can be restrictive in practice. In particular, in known convergence results, the local objective functions possessed by agents are required to satisfy a highly restrictive bounded-gradient-dissimilarity condition. The paper demonstrates convergence to the set of global minima while relaxing this key assumption.
\end{abstract}


\thispagestyle{plain}
\markboth{}{}

\section{Introduction}
In this paper we are interested in minimizing the nonconvex function 
\begin{equation} \label{eq:U-def}
U(x) = \sum_n U_n(x),
\end{equation}
where each $U_n:\R^d\to\R$ is smooth but possibly nonconvex. Optimization problems of this form are readily found in applications in machine learning and signal processing (see \cite{di2016next,nedic2009distributed,rabbat2004distributed}  and references therein).

In typical machine learning applications, such as empirical risk minimization, the objective function is naturally decomposed into the sum form \eqref{eq:U-def} by separating data into $N$ sets to generate each subfunction $U_n$.
In applications such as the internet of things (IoT) and sensor networks, data is inherently distributed among nodes of a network, giving rise to a natural decomposition of the form \eqref{eq:U-def}.
This motivates the development of distributed (i.e., network-based) algorithms for nonconvex optimization 
where data is stored locally at nodes of a network, and nodes (or agents) exchange algorithm-relevant information only with neighboring agents via an overlaid communication graph.

Much recent work on nonconvex optimization has focused on computation of local minima; e.g., \cite{murray2019revisiting,lee2016gradient,lee2017first,dauphin2014identifying} study convergence to local minima in centralized settings, and \cite{bianchi2012convergence,swenson2019distributed,di2016next,sun2016distributed,lu2019pa,tatarenko2017non},  study convergence to local minima (or critical points) in distributed settings.

A popular gradient-based technique for computing global minima of nonconvex functions consists of using gradient descent dynamics plus appropriately controlled Gaussian noise used for exploration \cite{welling2011bayesian,gidas1985global}. Such dynamics are often referred to as Langevin dynamics. It has been well established that these dynamics, in their classical form, converge asymptotically to global optima \cite{chiang1987diffusion,gelfand1991recursive,gidas1985global}. Motivated by applications in machine learning, recent work has focused on characterizing various aspects of Langevin dynamics including hitting time to local minima \cite{zhang2017hitting}, as well as local minima escape time and recurrence time \cite{tzen2018local,gao2018breaking}. Recent results in \cite{raginsky2017non,xu2018global} give finite time convergence guarantees for the Langevin dynamics in nonconvex problems.

In this paper we consider the following distributed algorithm for global optimization of \eqref{eq:U-def}
\begin{align} \label{eq:update-simple}
\vx_n(t+1) = & \vx_n(t) - \beta_t\sum_{\ell\in\Omega_n} (\vx_n(t) - \vx_\ell(t))\\
& \quad \quad - \alpha_t \nabla U_n(\vx_n(t)) + \gamma_t\vw_n(t).
\end{align}
These dynamics were originally introduced in \cite{swenson2019annealing} (see also \cite{swenson2019distributed}), and may be viewed as a distributed variant of the (centralized) global optimization algorithm considered in \cite{gelfand1991recursive}. More generally, these dynamics may be viewed as a distributed discrete-time analog of the continuous-time Langevin diffusion \cite{swenson2019annealing,chiang1987diffusion,gelfand1991recursive}. 
The work differs from current work on distributed nonconvex optimization, including \cite{bianchi2012convergence,swenson2019distributed,di2016next,sun2016distributed,tatarenko2017non,lu2019pa}, in that we study convergence to global optima, and more generally, in that we study distributed Langevin-type dynamics which have been shown to have desirable nonasymptotic properites in centralized settings \cite{zhang2017hitting,tzen2018local,gao2018breaking,raginsky2017non,xu2018global}.\footnote{We note that the process \eqref{eq:update-simple} uses decaying weight parameters, while recent work \cite{zhang2017hitting,tzen2018local,gao2018breaking,raginsky2017non,xu2018global} generally considers fixed-weight-parameter processes.}

While the results of \cite{swenson2019annealing} are promising, convergence of \eqref{eq:update-simple} to the set of global minima of $U$ was demonstrated only under restrictive assumptions.
In particular, it was assumed in \cite{swenson2019annealing} that agents' local objective functions satisfy the following \emph{bounded-gradient-dissimilarity} condition (see Assumption 2 in \cite{swenson2019annealing}):
$$
\sup_{x\in\R^d}\|\nabla U_n(x) - \nabla U(x)\| < \infty, \quad \forall  n=1,\ldots,N.
$$
This assumption, though asymptotic in nature, is highly restrictive; e.g., it is violated in the simple case that $d=1$ and $U_n(x) = c_nx^2$, $c_n\in\R$, $c_n\not= c_\ell$, for some $ n,\ell \in\{1,\ldots,N\}.$

The main result of this paper is to demonstrate convergence of \eqref{eq:update-simple} while relaxing this key assumption. In particular, in lieu of assuming bounded gradient-dissimilarity, we will assume the much weaker condition that each agent's local objective function is individually coercive and the gradient is radially nondecreasing (see Assumption \ref{ass:coercive} below). Aside from changing this key assumption, we retain the remaining assumptions used in \cite{swenson2019annealing} and prove convergence in probability to the set of global minima.\footnote{We note that, to simplify the presentation, we do assume here that the communication graph is time invariant (Assumption \ref{ass:conn}). However, the analysis readily extends to the time-varying case.} The main result of the paper is found in Theorem \ref{thrm:main-result}.

The remainder of the paper is organized as follows: Section \ref{sec:main-result} presents our assumptions and main result, and Section \ref{sec:conv-analysis} provides the convergence analysis for our main result. We remark that in the paper we use standard notational conventions, identical to those used in \cite{swenson2019annealing}.

\section{Assumptions and Main Result} \label{sec:main-result}
We now review our assumptions and present our main result.

We will assume that the local objective functions $U_n$, $n=1,\ldots,N$ satisfy the following two assumptions.
\begin{assumption}
\label{ass:Lip} $U_n(\cdot)$ is $C^2$ and has Lipschitz continuous gradient, i.e., there exists $K>0$ such that
\begin{align}
\left\|\nabla U_{n}(x)-\nabla U_{n}(\acute{\mathbf{x}})\right\|\leq K\left\|x-\acute{x}\right\|, \quad \forall n.
\end{align}
\end{assumption}

\begin{assumption} \label{ass:coercive}
$U_n(\cdot)$ is coercive, i.e.,
$$U_n(x)\to\infty \mbox{ as } \|x\|\to\infty.$$
Moreover, for some constant $C_1>0$, the gradient of $U_n$ satisfies $\langle x, \nabla U_n(x) \rangle \geq 0$ for $\|x\| \ge C_1$.
\end{assumption}

Assumptions \ref{ass:GM_1}--\ref{ass:GM_3} pertain to the sum function \eqref{eq:U-def}.
\begin{assumption}
\label{ass:GM_1} $U:\R^d\to\R$ satisfies
\begin{enumerate}
\item[(i)] $\min_x U(x) = 0$,
\item[(ii)] $\inf_x (|\nabla U(x)|^2 - \Delta U(x) ) > -\infty$.
\end{enumerate}
\end{assumption}
\begin{assumption}
\label{ass:GM_2} For $\e>0$ let
$$
d\pi^\e(x) = \frac{1}{Z^\e}\exp\left(-\frac{2U(x)}{\e^2} \right)\dx,
$$
where
$
Z^\e= \int\exp\left(-\frac{2U(x)}{\e^2} \right)\dx,
$
where $d\pi^\e$ denotes the Radon-Nikodym derivative of $\pi^\e$ taken with respect to the Lebesgue measure.
Assume
$U$ is such that $\pi^\e$ has a weak limit $\pi$ as $\e\to 0$.
\end{assumption}
\begin{assumption}
\label{ass:GM_3}
The gradient $\nabla U(x)$ satisfies the following conditions
\begin{itemize}
  \item [(i)] $\liminf_{\|x\|\to\infty}\langle \frac{\nabla U(x)}{\|\nabla U(x)\|}, \frac{x}{\|x\|} \rangle \geq C(d)$,
$C(d) = \left( \frac{4d-4}{4d-3} \right)^{\frac{1}{2}}$
  \item [(ii)] $\liminf_{\|x\|\to\infty} \frac{\|\nabla U(x)\|}{\|x\|} > 0$
  \item [(iii)] $\limsup_{\|x\|\to\infty} \frac{\|\nabla U(x)\|}{\|x\|} < \infty$
\end{itemize}
\end{assumption}

The remaining assumptions pertain to the process \eqref{eq:update}. For $t\geq 1$, let $\vx_t$ denote the $Nd$-dimensional vector stacking $(\vx_n(t))_{n=1}^N$, and let $\calF_{t} := \sigma(\{\vx_{s},\}_{s=1}^t,\{\bxi_s,\vw_s\}_{s=1}^{t-1},)$ so that $\{\calF_{t}\}$ denotes the natural filtration associated with \eqref{eq:update-simple}.

\begin{assumption}
\label{ass:conn}
There exists a communication graph $G = (V,E)$ over which agents may exchange information with neighboring agents. The graph $G$ is undirected and connected.
\end{assumption}
\begin{assumption}
\label{ass:grad_noise} The sequence $\{\bxi_{t}\}$ is $\{\mathcal{F}_{t+1}\}$-adapted and there exists a constant $B>0$ such that
\begin{align}
\label{ass:grad_noise1}
\mathbb{E}[\bxi_{t}~|~\calF_{t}]=0~~\mbox{and}~~\mathbb{E}[
\|\bxi_{t}\|^{2}~|~\calF_{t}]<B
\end{align}
for all $t\geq 0$.
\end{assumption}
\begin{assumption}
\label{ass:gauss}
For each $n$, the sequence $\{\mathbf{w}_{n}(t)\}$ is a sequence of i.i.d. $d$-dimensional standard Gaussian vectors with covariance $I_d$ and with $\mathbf{w}_{n}(t)$  being independent of $\calF_{t}$ for all $t$. Further, the sequences $\{\mathbf{w}_{n}(t)\}$ and $\{\mathbf{w}_{l}(t)\}$ are mutually independent for each pair $(n,l)$ with $n\neq l$.
\end{assumption}
\begin{assumption}
\label{ass:weights} The sequences $\{\alpha_{t}\}$, $\{\beta_{t}\}$, and $\{\gamma_{t}\}$ satisfy
\begin{align}
\label{eq:weights}
\alpha_{t}=\frac{c_{\alpha}}{t},~~\beta_{t}=\frac{c_{\beta}}{t^{\tau_{\beta}}},~~\gamma_{t}=\frac{c_{\gamma}}{t^{1/2}\sqrt{\log\log t}},~~~\mbox{for $t$ large},
\end{align}
where $c_{\alpha},c_{\beta},c_{\gamma}>0$ and $\tau_{\beta}\in (0,1/2)$.
\end{assumption}

Assumptions \ref{ass:Lip}--\ref{ass:coercive} ensure that agents can reach consensus. Assumptions \ref{ass:GM_1}--\ref{ass:GM_3} ensure that a global minimum of $U$ can be found by the annealing process. Assumptions \ref{ass:conn}--\ref{ass:weights} ensure that algorithmic parameters, including noise terms, weight parameters, and graph connectivity, are adequately chosen.

The key novelty of the paper lies in assuming Assumption \ref{ass:coercive} rather than bounded gradient dissimilarity (Assumption 2 in \cite{swenson2019annealing}). 
Assumption \ref{ass:GM_3} assumes similar conditions hold for the gradient of the \emph{sum} function, and is required to obtain convergence to global minima. Assumption \ref{ass:coercive} is only a slight strengthening of this assumption, applied individually to each local objective function. 

\subsection{Main Convergence Result}

Let
\begin{equation} \label{eq:x-bar-def}
\bar \vx_t := \frac{1}{N} \sum_{n=1}^N \vx_n(t).
\end{equation}

\begin{theorem}\label{thrm:main-result}
Suppose Assumptions \ref{ass:Lip}--\ref{ass:weights} hold and let $\{\vx_n(t)\}$ satisfy \eqref{eq:update-simple}, with initial condition $\vx_n(1) = x_{1,n}$, $n=1,\ldots,N$.
Suppose further that $c_{\alpha}$ and $c_{\gamma}$ in Assumption~\ref{ass:weights} satisfy $c_{\gamma}^{2}/c_{\alpha}>C_{0}$, where $C_0$ is defined after (2.3) in \cite{gelfand1991recursive}.
Then, for any bounded continuous function $f:\mathbb{R}^{d}\to\mathbb{R}$ and for all $n=1,\ldots,N$, we have that
\begin{align}
\label{lm:avg_conv1}
\lim_{t\to\infty} \E(f(\vx_n(t))\vert \vx_n(1) = x_{1,n}) = \int f(x)d\pi(x),
\end{align}
where $\pi$ is as defined in Assumption \ref{ass:GM_1}.
\end{theorem}

The following is an immediate corollary of Theorem \ref{thrm:main-result}.
\begin{corollary}
Suppose the hypotheses of Theorem \ref{thrm:main-result} hold. Then for $n=1,\ldots,N$, $\vx_n(t)$ converges in probability to the set of global minima of $U$, i.e., $\lim_{t\to\infty} \P(\vx_n(t)\in S) = 1$.
\end{corollary}
\begin{proof}

Let $S := \arg\min_{x\in\R^d} U(x)$.
Let $f_i:\R^d\to\R$ be a sequence of continuous functions converging pointwise to the indicator function on $S$, denoted as $\ones_{S}$. Fix $n\in\{1,\ldots,N\}$ and let $\P_t$ denote the probability measure induced by $\vx_n(t)$ over $\R^d$; i.e., for a Borel set $B$ of $\R^d$, $\P_t(B)$ indicates the probability that $\vx_n(t)$ lies in $B$.
By Theorem \ref{thrm:main-result} we have $\lim_{t\to\infty} \int f_i(x) d\P_t(x) = \int f_i(x)d\pi(x).$
Taking the limit as $i\to\infty$ on both sides and then exchanging the order of the limits (justified by Fubini's theorem \cite{williams1991probability}) we get the desired result.
\end{proof}

\section{Convergence Analysis} \label{sec:conv-analysis}
In this section we will prove Theorem \ref{thrm:main-result}. 
Due to the relaxation of the bounded-gradient-dissimilarity assumption (Assumption 2 in \cite{swenson2019annealing}), the techniques used to prove convergence to consensus in \cite{swenson2019annealing} are no longer applicable and an alternative approach must be taken. This is the main challenge in proving Theorem \ref{thrm:main-result} under our relaxed assumptions. In Lemmas \ref{lemma:boundedness}--\ref{lemma-consensus} we prove that \eqref{eq:update-simple} achieves consensus. Given Lemma \ref{lemma-consensus}, Theorem \ref{thrm:main-result} readily follows from the techniques developed in \cite{swenson2019annealing}. In particular, Theorem \ref{thrm:main-result} follows immediately from Lemma \ref{lemma-consensus} below and Lemma 5 in \cite{swenson2019annealing} (see Lemma \ref{lemma-conv-to-min} below). 

Before proving convergence of \eqref{eq:update-simple} to consensus, we recall the following result from~\cite{kar2013distributed} that will be useful in our proof techniques.

\begin{lemma}[Lemma 4.3 in~\cite{kar2013distributed}]
\label{lm:mean-conv} Let $\{\mathbf{z}_{t}\}$ be an $\mathbb{R}_{+}$ valued $\{\calF_{t}\}$ adapted process that satisfies
\begin{equation}
\label{lm:mean-conv1}
\mathbf{z}_{t+1}\leq \left(1-r_{1}(t)\right)\mathbf{z}_{t}+r_{2}(t)V_{t}\left(1+J_{t}\right).
\end{equation}
In the above, $\{r_{1}(t)\}$ is an $\{\calF_{t+1}\}$ adapted process, such that for all $t$, $r_{1}(t)$ satisfies $0\leq r_{1}(t)\leq 1$ and
\begin{equation}
\label{lm:JSTSP2}
\frac{a_{1}}{(t+1)^{\delta_{1}}}\leq\mathbb{E}\left[r_{1}(t)~|~\calF_{t}\right]\leq 1
\end{equation}
with $a_{1}>0$ and $0\leq \delta_{1}< 1$. The sequence $\{r_{2}(t)\}$ is deterministic, $\mathbb{R}_{+}$ valued and satisfies $r_{2}(t)\leq a_{2}/(t+1)^{\delta_{2}}$ with $a_{2}>0$ and $\delta_{2}>0$.
Further, let $\{V_{t}\}$ and $\{J_{t}\}$ be $\mathbb{R}_{+}$ valued $\{\calF_{t+1}\}$ adapted processes with $\sup_{t\geq 0}\|V_{t}\|<\infty$ a.s. The process $\{J_{t}\}$ is i.i.d.~with $J_{t}$ independent of $\calF_{t}$ for each $t$ and satisfies the moment condition $\mathbb{E}\left[\left\|J_{t}\right\|^{2+\varepsilon_{1}}\right]<\kappa<\infty$ for some $\varepsilon_{1}>0$ and a constant $\kappa>0$. Then, for every $\delta_{0}$ such that
\begin{equation}
\label{lm:mean-conv5}
0\leq\delta_{0}<\delta_{2}-\delta_{1}-\frac{1}{2+\varepsilon_{1}},
\end{equation}
we have $(t+1)^{\delta_{0}}\mathbf{z}_{t}\rightarrow 0$ a.s. as $t\rightarrow\infty$.
\end{lemma}

We now consider the issue of convergence to consensus. 
Note that \eqref{eq:update-simple} may be expressed compactly as
\begin{equation} \label{eq:update}
\vx_{t+1} = \vx_t - \beta_t (L\otimes I_d)\vx_{t} - \alpha_t \left( \nabla \hat U(\vx_t) + \bxi_t \right)+ \gamma_t \vw_t,
\end{equation}
where $\hat U(\vx_t) = \sum_{n=1}^N U_n(\vx_n(t))$, $L$ denotes the graph Laplacian of $G$, and $\otimes$ denotes the Kronecker product.
The following lemma characterizes the growth rate of $\|\vx_t\|$.
\begin{lemma} \label{lemma:boundedness}
Suppose Assumptions \ref{ass:Lip}--\ref{ass:weights} hold and suppose $\{\vx_t\}$ satisfies \eqref{eq:update}. Then for every $\eta > 1/2$, it holds with probability 1 that
$$
\sup_{t\geq 1} \frac{\|\vx_t\|}{t^\eta} < \infty.
$$
\end{lemma}
\begin{proof}
Using \eqref{eq:update} and letting $Q = (L\otimes I_d)$ we have
\begin{align}
\|\vx_{t+1}\|^2 = &\|(I-\beta_t Q)\vx_t\|^2 + \alpha_t^2 \|\nabla \hat U(\vx_t)\|^2\\
& +\alpha_t^2\| \bxi_t\|^2\ + \gamma_t^2 \|\vw_t\|^2\\
& + 2\alpha_t \langle (I-\beta_tQ)\vx_t, -\nabla \hat U(\vx_t) - \bxi_t \rangle\\
& + 2\gamma_t \langle (I-\beta_t Q\vx_t), \vw_t\rangle\\
& + 2\alpha_t\gamma_t\langle -\nabla \hat U(\vx_t) - \bxi_t , \vw_t\rangle\\
& + 2\alpha_t^2\langle \nabla \hat U(\vx_t),\bxi_t \rangle.
\label{eq:lemma1-eq1}
\end{align}

Taking the conditional expectation with respect to $\mathcal{F}_t$ on both sides of \eqref{eq:lemma1-eq1} and using Assumptions \ref{ass:grad_noise}--\ref{ass:gauss} we have
\begin{align}
\E(\|\vx_{t+1}\|^2\vert\calF_t) = & \|(I-\beta_t Q)\vx_t\|^2 + \alpha_t^2 \|\nabla \hat{U}(\vx_t)\|^2\\
& + \alpha_t^2\E(\|\bxi_t\|^2\vert\calF_t) + \gamma_t^2\\
& -  2\alpha_t\langle (I-\beta_tQ)\vx_t, \nabla \hat U(\vx_t) \rangle \\
& \leq (1-\beta_t\lambdaMin)\|\vx_t\|^2 + \alpha_t^2 K\|\vx_t\|^2\\
& + \alpha_t^2B + \gamma_t^2 + 2 K \alpha_t \beta_t \| \vx_t \|^2 \\
&  - 2 \alpha_t \langle \vx_t, \nabla \hat{U}(\vx_t) \rangle,
\label{eq:lemma1-eq2}
\end{align}
where, to obtain the inequality, we repeatedly apply Assumption \ref{ass:Lip} and $\lambdaMin$ denotes the smallest eigenvalue of $Q$ (in this case 0). We proceed by lower bounding $\langle \vx_t, \nabla \hat{U}(\vx_t) \rangle$. We can expand the inner product as $$
\langle \vx_t, \nabla \hat{U}(\vx_t) \rangle = \sum\limits_{i = 1}^N \langle \vx_i(t), \nabla U_i(\vx_i(t)) \rangle.
$$
If $\|\vx_i(t)\| \ge C_1$, then by Assumption \ref{ass:coercive}, $\langle \vx_i(t), \nabla U_i(\vx_i(t)) \rangle \ge 0$. Else, Assumption \ref{ass:Lip} implies $$
\langle \vx_i(t), \nabla U_i(\vx_i(t)) \rangle \ge - K \| \vx_i(t) \|^2 \ge - K C_1.
$$
It follows that $\langle \vx_t, \nabla \hat{U}(\vx_t) \rangle \ge - N K C_1$, so we obtain from \eqref{eq:lemma1-eq2} that for $t$ sufficiently large, 
\begin{equation}\label{eq:lemma1-eq3}
\E \left( \| \vx_{t+1} \|^2 \right) \le (1 + 3 K \alpha_t \beta_t ) \| \vx_t \|^2  + 2N K C_1.
\end{equation} 
Fix $\eta > 1$. Dividing both sides of \eqref{eq:lemma1-eq3} by $(t+1)^\eta$ we have
\begin{equation} \label{eq:almost-SM}
\frac{\E(\|\vx_{t+1}\|^2\vert\calF_t)}{(t+1)^\eta} \leq
(1+3 K \alpha_t \beta_t) \cdot \left( \frac{t}{t+1} \right)^\eta \frac{\|\vx_t\|^2}{t^\eta} + \frac{2NKC_1}{(t+1)^\eta}.
\end{equation}
For $t$ large,
\begin{equation}
1 + 3 K \alpha_t \beta_t \le 1 + \frac{\eta}{t} \le \left( \frac{t+1}{t} \right)^\eta,
\end{equation}
where the second inequality is a consequence of Bernoulli's inequality as $\eta \ge 1$. Substituting the above in \eqref{eq:almost-SM} gives
\begin{equation} \label{eq:almost-SM-2}
\E \left( \frac{ \| \vx_{t+1} \|^2}{ (t + 1)^\eta } \right) \le \frac{ \| \vx_t \|^2 }{ t^\eta } + \frac{ 2NKC_1 }{(t+1)^\eta}.
\end{equation}
Let $\zeta_t = \frac{2NKC_1}{t^\eta}$,  let
$$
V_t = \frac{\|\vx_t\| ^2}{t^\eta}- \sum_{s=1}^t \zeta_s,
$$
and note that $\sum_{s=1}^\infty \zeta_s < \infty$.
By \eqref{eq:almost-SM} we have
\begin{align}
\E(V_{t+1}\vert\calF_t) & = \E\left(\frac{\|\vx_{t+1}\|^2}{(t + 1)^\eta}\vert\calF_{t} \right) - \sum_{s=1}^{t+1} \zeta_s\\
& \leq \frac{\|\vx_{t}\|^2}{t^\eta} - \sum_{s=1}^{t} \zeta_s
= V_{t},
\end{align}
so that $\{V_t\}$ is a supermartingale. By Doob's supermartingale inequality \cite{williams1991probability}, we have $\P \left(\sup_{t\geq 1} V_t \geq c \right) \leq \frac{\E(V_1)}{c}, $
for any $c>0$. Sending $c\to\infty$ we see that $\P \left(\sup_{t\geq 1} V_t < \infty \right) = 1.$
Since $\sum_{t\geq 1} \zeta_t < \infty$, this implies that $\P\left(\sup_{t\geq 1} \frac{\|\vx_t\|^2}{t^\eta} < \infty \right) =0,$
which is equivalent to the desired result.
\end{proof}

The following lemma establishes that agents achieve asymptotic consensus.
\begin{lemma} \label{lemma-consensus}
Suppose Assumptions \ref{ass:Lip}--\ref{ass:weights} hold. Then with probability 1, for every $\tau\in[0,1/2-\tau_\beta)$ there holds
$$
\lim_{t\to\infty} t^\tau \|\vx_n(t) - \bar \vx(t)\| =0.
$$
\end{lemma}
\begin{proof}
Let $R$ be an orthonormal matrix diagonalizing $L\otimes I_d$. Let $Q = R^T(L\otimes I_d)R$ and without loss of generality assume that the diagonal entries of $Q$ are arranged so that $Q = \text{diag}(\textbf{0},\hat Q)$, where $\hat Q\in \R^{(N-1)d\times(N-1)d}$ is positive definite. (Since the graph $G$ is connected, the nullspace of $L\otimes I_d$ has dimension $d$.) Letting $\vy_t = R\vx_t$, the recursion \eqref{eq:update} is equivalent to
\begin{equation} \label{eq:gen-recursion}
\vy_{t+1} = \vy_t - \beta_t Q\vy_t - \alpha_t (\nabla h(\vy_t) + \bxi_t) + \gamma_t \vw_t,
\end{equation}
$h(y) := U(R^{T}y)$ and $\bxi_t$ and $\vw_t$ satisfy Assumptions \ref{ass:grad_noise}--\ref{ass:gauss}.
Let $\vy_t$ be decomposed as
$$
\vy_t =
\begin{pmatrix}
\bar \vy_t\\
\vy_t^\perp
\end{pmatrix},
$$
where $\bar \vy_t \in \R^d$ and $\vy_t^\perp \in \R^{(N-1)d}$.
From \eqref{eq:gen-recursion} we have
$$
\vy_{t+1}^\perp  = \vy_{t}^\perp -\beta_t \hat Q\vy_t^\perp -\alpha_t \nabla h^\perp(\vy_t) - \alpha_t\bxi_t^\perp + \gamma_t\vw_t^\perp,
$$
By Assumption \ref{ass:Lip} we have $\|\nabla h(\vy_t)\| \leq C_2+ K\|\vy_t||$ for some $C_2 \ge 0$. Hence, letting $\eta\in(1/2,1 - \tau_\beta)$ we have
\begin{align}
\|\vy_{t+1}^\perp\| & \leq \|\vy_t^\perp - \beta_t \hat Q\vy_t^\perp\| +  \alpha_t \|\nabla h^\perp(\vy_t)\|\\
 & + \alpha_t \|\bxi_t^\perp\|+ \gamma_t\|\vw_t^\perp\|\\
\leq & \|(I-\beta_t \hat Q)\vy_t^\perp\| +  \alpha_t K(\|\vy_t^\perp\| + \|\bar \vy_t\|) + \alpha_t C_2\\
& + \alpha_t \|\bxi_t^\perp\| + \gamma_t\|\vw_t^\perp\|\\
\leq &(1-\beta_t \hat Q)\|\vy_t^\perp\| +  \alpha_t K\|\vy_t^\perp\| + \alpha_t C_2\\
& + \alpha_t Kt^{\eta} + \alpha_t\|\bxi_t^\perp\| + \gamma_t\|\vw_t^\perp\|\\
\leq & (1-\bar \beta_t)\|\vy_t^\perp\| + \alpha_t K t^{\eta} + \alpha_t C_2\\
& + \alpha_t \|\bxi_t^\perp\| + \gamma_t\|\vw_t^\perp\|
\end{align}
where $\bar\beta_t = \beta_t \lambdaMin - \alpha_tK$, $\lambdaMin$ is the smallest eigenvalue of $\hat Q$, and the third inequality follows from Lemma \ref{lemma:boundedness}.
Using the Borel-Cantelli Lemma \cite{williams1991probability}, it is straightforward to verify that $\frac{1}{t} \|\bxi_t\| \to 0$ as $t\to\infty$, almost surely (see, e.g., \cite{swenson2019annealing}, Lemma 1). Moreover, by Assumption \ref{ass:gauss}, $\|\vw_t\|$ possesses moments of all order. Thus, there exist $\{\calF_t\}$ adapted processes $V_t$ and $J_t$ such that
$$
\alpha_tK t^{\eta}+\alpha_t \|\bxi_t^\perp\| + \gamma_t\|\vw_t^\perp\| \leq t^{-(1 -\eta) }V_t(1+J_t),
$$
where $\sup_{t\geq 1} V_t < \infty$ almost surely, and $J_t$ possesses moments of all orders. Substituting this into the previous chain of inequalities we have
$$
\|\vy_{t+1}\| \leq (1-\bar \beta_t)\|\vy_t^\perp\| + t^{-(1 - \eta)}V_t(1+J_t),
$$
which fits the template of Lemma \ref{lm:mean-conv} since $1 - \eta > \tau_\beta$. By Lemma \ref{lm:mean-conv} we have $t^{\tau}\|\vy_t\| \to 0$ for every $\tau \in[0,1 - \eta -\tau_\beta)$. This holds for $\eta$ arbitrarily close to $\frac{1}{2}$, so the desired result follows. 
\end{proof}

Theorem \ref{thrm:main-result} now follows immediately from Lemma \ref{lemma-consensus} above and the following Lemma from \cite{swenson2019annealing}.\footnote{We remark that in Lemma 5 in \cite{swenson2019annealing} it is assumed that the agents' utilities satisfy a bounded-gradient-dissimilarity condition (Assumption 2 in \cite{swenson2019annealing}). This is only required in the proof Lemma 5 of \cite{swenson2019annealing} to ensure consensus is reached. Since consensus is reached under alternate assumptions here, this assumption is not needed.}
\begin{lemma}[\hspace{-.01em}\cite{swenson2019annealing}, Lemma 5] \label{lemma-conv-to-min}
Let $\{\vx_t\}$ satisfy the recursion \eqref{eq:update} and let $\{\bar \vx_t\}$ be given by \eqref{eq:x-bar-def}
with initial condition $x_1\in \R^d$. Let Assumptions \ref{ass:Lip}--\ref{ass:weights} hold.
Further, suppose that $c_{\alpha}$ and $c_{\gamma}$ in Assumption~\ref{ass:weights} satisfy, $c_{\gamma}^{2}/c_{\alpha}>C_{0}$, where $C_0$ is defined after Assumption~6 in \cite{swenson2019annealing}. Then, for any bounded continuous function $f:\mathbb{R}^{d}\to\mathbb{R}$, we have that $\lim_{t\to\infty} \E(f(\bar \vx_t)\vert \bar \vx_1 = x_1) = \int f(x)d\pi(x), $
where $\pi$ is as defined in Assumption \ref{ass:GM_1}.
\end{lemma}

\bibliographystyle{IEEEtran}
\bibliography{dist_glob_opt}

\end{document}